\documentclass[11pt]{amsart}
\usepackage{fullpage}
\usepackage{amsfonts}
\usepackage{amsmath}
\usepackage{amssymb}
\usepackage{amsthm}
\newtheorem{theorem}{Theorem}

\newtheorem{cor}[theorem]{Corollary}

\newtheorem{lemma}[theorem]{Lemma}
\newtheorem{prop}[theorem]{Proposition}

\newcommand{\norm}[1]{\left|\left|#1\right|\right|}
\newcommand{\sph}{\mathbb{S}^{n-1}}
\newcommand{\normi}[1]{\left|\left|#1\right|\right|_{\infty}}
\newcommand{\rn}{\mathbb{R}^{n}}
\newcommand{\p}{Pos_{n,2d}}

\newcommand{\sq}{Sq_{n,2d}}
\newcommand{\ip}[2]{\langle #1,#2 \rangle }
\newcommand{\st}{\hspace{2mm} \mid \hspace{2mm}}

\newcommand{\spht}{\mathbb{S}^{n-1} \times \mathbb{S}^{n-1}}
\newcommand{\ipe}{\langle \cdot, \cdot \rangle}
\newcommand{\foral}{\hspace{2.5mm} \textnormal{for all} \hspace{2mm}}
\newcommand{\po}{\textnormal{o}}
\newcommand{\normil}{\left|\left|\hspace{.3mm}l\right|\right|_{\infty}}
\newcommand{\hsp}{\hspace{.5mm}}
\numberwithin{theorem}{section}
\numberwithin{equation}{section}
\begin{document}
\title{Convex Forms That Are Not Sums of Squares}
\author{Grigoriy Blekherman}
\begin{abstract}
An orbitope is the convex hull of an orbit of a point under the action of a compact group. We derive bounds on volumes of sections of polar bodies of orbitopes, extending methods developed in \cite{Sasha2}. As an application we realize the cone of convex forms as a section of the cone of nonnegative bi-homogeneous forms and estimate its volume. A convex form has to be nonnegative, but it has not been previously shown that there exist convex forms that are not sums of squares. Combining with the bounds of \cite{Me2} we show that if the degree is fixed then the cone of convex forms has asymptotically same size as the cone of nonnegative forms and it is significantly larger asymptotically than the cone of sums of squares. This implies existence of convex forms that are not sums of squares, although there are still no known examples.
\end{abstract}
\maketitle

\section{Introduction and Results.}
Let $G$ be a compact group acting on a vector space $V$ endowed with a $G$-invariant inner product $\ip{\cdot}{\cdot}$. Let $v$ be a vector in $V$ and let $B=B(v)$ be the orbitope of $v$, i.e. the convex hull of the orbit of $v$:
$$B=B(v)=\textnormal{conv}\{gv \mid g \in G\}.$$

\noindent We will assume that the orbit of $v$ spans $V$ affinely. If it does not we can always restrict to the affine hull of the orbit.

We will be working with the convex bodies $B^{\textnormal{o}}=B^{\textnormal{o}}(v)$ that are dual to orbitopes:
$$B^{\textnormal{o}}=B^{\textnormal{o}}(v)=\{l \in V^* \st l(gv) \leq 1 \foral g \in G\}.$$

\noindent We also think of $B^\po$ as a convex set in $V$ by identifying a vector $x \in V$ with the linear functional $l_x$ using the $G$-invariant inner product $\ipe$:
$$l_x(y) = \ip{x}{y} \foral y \in V.$$

\noindent Using this identification the definition of $B^{\textnormal{o}}$ translates to:
$$B^{\textnormal{o}}=B^{\textnormal{o}}(v)=\{w \in V \st \ip{w}{gv} \leq 1 \foral g \in G\}.$$

If we fix a point $v \in V$ then we can also think of linear functionals comprising $B^{\po}(v)$ as functions on the group $G$: $$l(g)=l(gv).$$

\noindent In this context we will use $\normil$ to denote the maximum absolute value of $l$ on $G$:

$$\normil=\max_{g \in G} |\hspace{.3mm} l(gv)|$$

\noindent and $\norm{\hspace{.3mm} l}_{2k}$ to denote the $L^{2k}$ norm of $l$:

$$\norm{\hspace{.3mm} l}_{2k}=\left(\int_G l^{2k}(gv) \hsp d\mu\right)^{\frac{1}{2k}},$$

\noindent where $\mu$ is the Haar probability measure on $G$.

We will derive bounds on volumes of sections of $B^\po$ with linear subspaces. Let $W$ be a subspace of $V$ of dimension $d_w$ and let $B^{\textnormal{o}}_W$ be the section of $B^{\textnormal{o}}$ with $W$:

$$B_W^{\textnormal{o}}=B^{\textnormal{o}} \cap W.$$

\noindent Let $S_W$ be the unit sphere and let $\Sigma_W$ be the unit ball in $W$ with respect to $\ipe$.

\subsection{General Bound.} We will prove the following bound on the volume of $B_W$:

\begin{theorem} \label{mainthm}
Let $V$ be a vector space endowed with an action of compact group $G$. For a vector $v \in V$ let $B^{\po}(v)$ be the dual to the orbitope of $v$.
Suppose that for some integer $k$ and $\alpha_k > 0$ we can bound the $L^{\infty}$ norm of any linear functionals $l \in V^*$  on $G$ with the $L^{2k}$ norm of $l$:
\begin{equation*}
\normi{\hspace{.3mm} l} \leq \alpha_k \norm{\hspace{.3mm}l}_{2k}.
\end{equation*}
Then we have the following bound for the volume of $B^{\po}_W(v)$:
\begin{equation}\label{mainbound}
\left(\frac{\textnormal{Vol} \hspace{.5mm}B^{\po}_W(v)}{\textnormal{Vol} \hspace{.5mm} \Sigma_W} \right)^{1/d_w} \geq \alpha_k^{-1} \sqrt{\frac{d_w}{2k\ip{v}{v}}}.
\end{equation}
\end{theorem}

The bound of \eqref{mainbound} requires some explanation. Let $v^{\otimes k}$ denote the $k$-th tensor power of $v$ and let $D_k(v)$ be the dimension of the span of $(gv)^{\otimes k}$ taken for all $g \in G$. It was shown in \cite{sasha3} that for any linear functional $l$ and all $k$ we can take $\alpha_K=\left(D_K(v)\right)^{1/2k}$.

In particular, we know that for any $v \in V$ the vector $v^{\otimes k}$ is contained in the $k$-th symmetric power $\textnormal{Sym}^k V$ of $V$. If we use the full dimension of $\textnormal{Sym}^k V$ as a bound for $D_k(v)$ then it is not hard to check that we get asymptotically same dependence on the dimension of $W$ in \eqref{mainbound} independent of $k$. In other words, the bound that we obtain are in some sense explained by ellipsoids, which is what we get if we choose $k=1$.

However, if a point $v$ has a lot of symmetries, so its stabilizer is a "large" subgroup of $G$, then it is possible to do better. The tensors $(gv)^{\otimes k}$ will have a lot of symmetries that come from the stabilizer and will span a smaller subspace. This is indeed what happens for nonnegative forms and nonnegative multi-homogeneous forms. We will briefly explains this phenomenon here with full details given in \cite{sasha3}.

\subsection{Nonnegative Forms.}Let $P_{n,2d}$ be the vector space of forms in $n$ variables of degree $2d$. We let $G=SO(n)$ act on $P_{n,2d}$ by rotating the variable space of the forms. Let $Pos_{n,2d}$ be the cone of nonnegative forms in $P_{n,2d}$:

$$Pos_{n,2d}= \left\{p \in P_{n,2d} \st p(x) \geq 0 \foral x \in \rn \right\}.$$
Let $Lf_{n,2d}$ be the cone of sums of $2d$-th powers of linear forms:

$$Lf_{n,2d}=\left\{p \in P_{n,2d} \st p=\sum_i l_i^{2d} \hspace{2.5mm} \textnormal{for some} \hspace{2.5mm} l_i \in P_{n,1}\right\}.$$

It can be shown that with a natural choice of $G$-invariant inner product the cones $\p$ and $Lf_{n,2d}$ are dual to each other \cite{Rez3}. The cone $Lf_{n,2d}$ is the conical hull of an orbit of one point, for example, we can take $v=x_1^{2d}$. The point $x_1^{2d}$ has a large stabilizer, namely the copy of $SO(n-1)$ that fixes the first standard vector $e_1$ in $\rn$. If we look at the $k$-th tensor powers, it is easy to see that points of the form $(gv)^{\otimes k}$ lie in the vector space of forms of degree $2kd$ and this vector space has much smaller dimension than the $k$-th symmetric power of $P_{n,2d}$. See \cite{sasha3} for full details.

Let $Sq_{n,2d}$ the convex cones of sums of squares in $P_{n,2d}$:

$$\sq=\left\{p \in P_{n,2d} \st p=\sum_i q_i^2 \hspace{2.5mm} \textnormal{for some} \hspace{2.5mm} q_i \in P_{n,d}\right\}.$$

In order to talk about volume of a cone we need to take a compact section with a hyperplane. Let $M'_{n,2d}$ be the hyperplane of forms that have integral $1$ on $\sph$:

$$M'_{n,2d}=\left\{p \in P_{n,2d} \st \int_{\sph} p \hsp d \sigma=1\right\},$$

\noindent and let $Pos'_{n,2d}$ and $Sq'_{n,2d}$ be the sections of $\p$ and $\sq$ with $M'_{n,2d}$:

$$Pos'_{n,2d}=\p \cap M'_{n,2d} \hspace{3mm} \textnormal{and} \hspace{3mm} Sq'_{n,2d}=\p \cap M'_{n,2d}.$$

The collapsing of dimension of span of $(gv)^{\otimes k}$ allowed us to derive some of the volume bounds for $Pos'_{n,2d}$ and $Sq'_{n,2d}$ given in \cite{Me2}. We used the following $L^2$ inner product on $P_{n,2d}$:

$$\ip{p}{q}_2=\int_{\sph} pq \hsp d\sigma,$$
\noindent where $\sigma$ is the rotation invariant probability measure on $\sph$.

Let $\Sigma_2$ be the unit ball in $M'_{n,2d}$ with respect to $\ipe_2$ and let $D_M$ be the dimension of $M$. It was shown in \cite{Me2} that if the degree $2d$ is fixed then
\begin{eqnarray*}
\left(\frac{\textnormal{Vol} \hsp Pos'_{n,2d}}{\textnormal{Vol} \hsp \Sigma_2}\right)^{1/D_M}=&\Theta(n^{-1/2})\\
\left(\frac{\textnormal{Vol} \hsp Sq'_{n,2d}}{\textnormal{Vol} \hsp \Sigma_2}\right)^{1/D_M}=&\Theta(n^{-d+1/2}).
\end{eqnarray*}
In particular when the degree $2d$ is fixed and at least $4$ we see that the volume of $Pos'_{n,2d}$ grows asymptotically much faster than the volume of $Sq'_{n,2d}$.
\subsection{Convex Forms.}
Let $K_{n,2d}$ be the convex cone of forms that are convex on $\rn$. If a form $p$ is convex then it must be nonnegative: if $p(v) <0$ for some $v \in \rn$ then $p$ restricted to the ray $\lambda v$ ($\lambda >0$) is concave. Therefore we see that $K_{n,2d}$ is contained in the cone of nonnegative forms $\p$.

To a form $p \in P_{n,2d}$ we associate its Hessian $H_p=(h_{ij})$ which is the matrix of second derivatives of $p$:

$$h_{ij}=\frac{\partial^2 p}{\partial x_i \partial x_j}.$$

A form $p$ is convex if and only if its Hessian is positive semi-definite at every point in $\rn$. In other words the form $B_p$ given by

$$B_p(x,y)=y^T H_p(x) y$$

\noindent is nonnegative for every $(x,y) \in \mathbb{R}^{2n}$. The mapping of $p$ to $B_p$ is clearly a linear operation. We observe that $B_p$ is a form in $2n$ variables of degree $2d$ and it is bi-homogeneous in $x$ and $y$. It is quadratic in the $y$ variables and of degree $2d-2$ in the $x$ variables.


Let $Bi_{2n,2d}$ be the vector space of bi-homogeneous forms in $2n$ variables, with two classes of variables $x$ and $y$ consisting of $n$ variables each. We require that the forms have degree $2d-2$ in the $x$ variables and degree $2$ is the $y$ variables. Let $D_{bi}$ be the dimension of $Bi_{2n,2d}$. It is not hard to show that
$$D_{bi}=\binom{n+2d-3}{2d-2}\binom{n+1}{2}.$$


Since we do not want to mix $x$ and $y$ variables it is natural to think of $G=SO(n) \times SO(n)$ acting on the forms in $Bi_{2n,2d}$, with each copy of $SO(n)$ acting on $x$ and $y$ variables separately. We will restrict forms in $Bi_{2n,2d}$ to $\sph \times \sph$ since $\sph \times \sph$ is an orbit of a unit vector in $\mathbb{R}^{2n}$ under the group action we described.

Any form $b \in Bi_{2n,2d}$ can be written as $$b=y^TM(x)y$$ for some symmetric matric $M$ whose entries are forms of degree $2d-2$ in $x$. Let $Pos_{bi}$ be the convex cone of nonnegative bi-homogeneous forms in $Bi_{2n,2d}$. We can identify the cone $K_{n,2d}$ with the section of $Pos_{bi}$ with the linear subspace of forms $B_p(x)=y^TH_p(x)y$ for some form $p \in P_{n,2d}$.


We will use the following inner product on $Bi_{2n,2d}$:

$$\ip{f}{g}=\int_{\spht} f \hspace{-.4mm} g \hspace{.5mm} d\sigma.$$

\noindent This induces the following "Hessian" inner product on $P_{n,2d}$:

$$\ip{p}{q}_H=\int_{\sph \times \sph} B_pB_q \hspace{.5mm} d\sigma.$$

This is clearly a positive definite quadratic form on $P_{n,2d}$ and it is also invariant under the natural action of $SO(n)$ on $P_{n,2d}$.

In order to talk about volume we take a section of $K_{n,2d}$ with the hyperplane $M'_{n,2d}$ of forms of integral $1$ on the unit sphere $\sph$ and call it  $K'_{n,2d}$:

$$K'_{n,2d}=\left\{p \in P_{n,2d} \st p \hspace{2.5mm} \textnormal{is convex and} \hspace{2.5mm} \int_{\sph} p \hsp d\sigma=1\right\}.$$


Similar to the case of homogeneous forms, its possible to see nonnegative bi-forms in $Bi_{2n,2d}$ as being dual to an orbitope; this is done in Section \ref{biformsasorbitope}. We can then apply volume bounds of Theorem \ref{mainthm} to $K'_{n,2d}$ since we identified $K_{n,2d}$ with the section of $Pos_{bi}$. We prove the following Theorem on the volume of $K'_{n,2d}$:


\begin{theorem}\label{convexthm}
\begin{equation*}
\left(\frac{\textnormal{Vol}\hsp K'_{n,2d}}{\textnormal{Vol} \hsp \Sigma_H}\right)^{1/D_M} \geq \frac{2d}{9e^2 \hsp \sqrt{2n\ln (2d+1)}}\sqrt{\frac{D_M}{D_{bi}}}.
\end{equation*}

\end{theorem}

When the degree $2d$ is fixed its easy to see that both $D_M$ and $D_{bi}$ have order $n^{2d}$ and therefore their ratio is bounded above by a constant. Thus we know that

$$\left(\frac{\textnormal{Vol}\hsp K'_{n,2d}}{\textnormal{Vol} \hsp \Sigma_H}\right)^{1/D_M} \hspace{5mm} \textnormal{is at least of the order} \hspace{5mm} n^{-1/2}.$$

\noindent We recall that $$\left(\frac{\textnormal{Vol} \hsp Pos'_{n,2d}}{\textnormal{Vol} \hsp \Sigma_2}\right)^{1/D_M}=\Theta\left(n^{-1/2}\right) \hspace{.7cm} \textnormal{and} \hspace{.7cm}
\left(\frac{\textnormal{Vol} \hsp Sq'_{n,2d}}{\textnormal{Vol} \hsp \Sigma_2}\right)^{1/D_M}=\Theta\left(n^{-d+1/2}\right).$$

\noindent However we are dividing by the volume of unit balls in different metrics.
We show in Section \ref{relmetrics} that if the degree is fixed then the Hessian inner product and the $L^2$ inner product are within a constant factor of each other and therefore
$$\left(\frac{\textnormal{Vol} \hsp \Sigma_2}{\textnormal{Vol}\Sigma_H}\right)^{1/D_M} \leq c(d),$$
for some number $c(d)$ depending on the degree only.
Therefore we see that the volume of $K'_{n,2d}$ is asymptotically of the same order as the volume of $Pos'_{n,2d}$, and it is asymptotically much larger than the volume of $Sq'_{n,2d}$ when the degree $2d$ is at least $4$.

We begin by establishing the volume bound on sections of the duals of orbitopes given in Theorem \ref{mainthm}.

\section{Volume Bound on Sections of $B^{\po}(v)$.}

Our bound on the volume of $B^\po(v)$ is derived in three steps. The first is to bound the volume of $B^\po(v)$ with an integral expression involving $L^{\infty}$ norms of linear functionals. This is done in Lemma \ref{volumeint}.

Next we replace the $L^{\infty}$ norms with $L^{2k}$ norms for an appropriate value of $k$. The value of $k$ depends on the representation of $G$. If we want the $2k$-th moments to be within a constant factor of the $L^{\infty}$ norm then it suffices to choose $k$ linear in the dimension of $V$. This is a sharp bound in general, but in some cases it is possible to do better. In our examples of nonnegative forms and convex forms we will indeed choose $k$ much lower than dimension of $V$. These ideas were developed in \cite{sasha3} we refer the reader to that paper for more details. This step is carried out in the proof of Theorem \ref{mainthm}.

The final step is to bound the resulting integral involving $2k$-th moments. The calculation is similar to \cite{Sasha2} Lemma 3.5 but we extend it to handle sections. This is done in Lemma \ref{2k-moments}.

\begin{lemma} \label{volumeint}
\begin{equation} \label{volumeinteq}
\left(\frac{\textnormal{Vol} \hspace{.5mm} B_W^{\po}}{\textnormal{Vol} \hspace{.5mm} \Sigma_W}\right)^{1/d_w} \geq \left(\int_{S_W} \normi{\hspace{.5mm}l_x} \hspace{.5mm} dx\right)^{-1}.
\end{equation}
\end{lemma}

\begin{proof}
Let $K \subset V$ be a convex body with $0$ in its interior. The gauge $Ga_K$ of $K$ is a function on $V$ that for a point $x \in V$ how much $K$ needs to be expanded to include $x$:
$$Ga_K(x)=\min \{\lambda \in \mathbb{R} \st x \in \lambda K \}.$$
By using polar coordinates we can write the following expression for the volume of $K$:
\begin{equation*}
\left(\frac{\textnormal{Vol} \hspace{.5mm} K}{\textnormal{Vol} \hspace{.5mm} \Sigma}\right)^{1/d}=\left(\int_S Ga_K^{-d}\right)^{1/d},
\end{equation*}
where $S$ is the unit sphere and $\Sigma$ is the unit ball in $V$.

The gauge of $B^\po_W$ is given by the maximum of the linear functional on the orbit of $v$:
\begin{equation*}
Ga_{B^\po_W}(x)=\max_{g \in G} l_x(gv)=\max_{g \in G} \ip{x}{gv}.
\end{equation*}
Thus we obtain the following expression for the volume of $B^\po$:
$$\left(\frac{\textnormal{Vol} \hspace{.5mm} B_W^{\po}}{\textnormal{Vol} \hspace{.5mm} \Sigma_W}\right)^{1/d_w}=\left(\int_{S_W}\max_{g \in G} l_x(gv)^{-d_w}\hspace{1mm} dx\right)^{1/d_w}
\geq \left(\int_{S_W}\normi{\hspace{.5mm}l_x}^{-d_w}\hspace{1mm} dx\right)^{1/d_w}.$$
Now we successively apply H\"{o}lder and Jensen inequalities to see that
$$\left(\int_{S_W}\normi{\hspace{.5mm}l_x}^{-d_w}\hspace{1mm} dx\right)^{1/d_w} \geq \int_{S_W}\normi{\hspace{.5mm}l_x}^{-1}\hspace{1mm} dx \geq \left(\int_{S_W}\normi{\hspace{.5mm}l_x}\hspace{1mm} dx \right)^{-1}.$$
\end{proof}

We now prove the lemma that bounds the average of $2k$-th moments over the unit sphere. We plan to approximate the $L^{\infty}$ norms in \eqref{volumeinteq} with $L^{2k}$ norms for an appropriate value of $k$ and then applying the following lemma.
\begin{lemma} \label{2k-moments}
Let $G$ be a compact group acting on vector space $V$ endowed with a $G$-invariant inner product $\ip{\cdot}{\cdot}$ and let $v$ be a vector in $V$. Let $W$ be a subspace of $V$ of dimension $d_w$ and let $S_W$ be the unit sphere in $W$ with respect to $\ip{\cdot}{\cdot}$. Then we have the following inequality bounding moments of linear functions on $S_W$:
\begin{equation*}
\int_{S_W} \left(\int_G \ip{x}{gv}^{2k} \hspace{.5mm} dg \right)^{\frac{1}{2k}} dx\leq \sqrt{\frac{2k\ip{v}{v}}{d_w}}.
\end{equation*}
\end{lemma}
\begin{proof}
Applying H\"{o}lder inequality we see that:

$$\int_{S_W} \left(\int_G \ip{x}{gv}^{2k} \hspace{.5mm} dg \right)^{\frac{1}{2k}} dx \leq \left(\int_{S_W} \int_G \ip{x}{gv}^{2k} \hspace{.5mm} dg \hspace{1mm} dx\right)^{\frac{1}{2k}}.$$

\noindent Exchanging the integrals we get:

\begin{equation}\label{exhangeint}
\int_{S_W} \int_G \ip{x}{gv}^{2k} \hspace{.5mm} dg \hspace{1mm} dx = \int_G \int_{S_W} \ip{x}{gv}^{2k} \hspace{.5mm} dx \hspace{1mm} dg.
\end{equation}

\noindent Now we observe that the inner integral $$\int_{S_W} \ip{x}{gv}^{2k} \hspace{.5mm} dx$$ is a $2k$-th power of a linear form integrated over a unit sphere. Let $gv_w$ be the orthogonal projection of $gv$ on $W$. Then we know that

\begin{equation*}
\int_{S_W} \ip{x}{gv}^{2k} \hspace{.5mm} dx = \ip{gv_w}{gv_w}^{k}\frac{\Gamma(d_w/2)\Gamma(k+1/2)}{\sqrt{\pi}\hspace{.5mm}\Gamma(k+d_w/2)}.
\end{equation*}

\noindent We know that orthogonal projection does not increase the norm and therefore $$\ip{gv_w}{gv_w} \leq \ip{gv}{gv}=\ip{v}{v},$$ where the second equality follows by $G$-invariance of the inner product. Using the inequalities $\Gamma(k+1/2)\leq \Gamma(k+1) \leq k^k$ and $$\frac{\Gamma(d_w/2)}{\Gamma(k+d_w/2)}=\frac{1}{(d_w/2)(d_w/2+1)\ldots (d_w/2+k-1)} \leq (d_w/2)^{-k}$$ we see that

\begin{equation*}
\int_{S_W} \ip{x}{gv}^{2k} \hspace{.5mm} dx \leq \left(\frac{2k}{d_w}\right)^k \ip{v}{v}^k.
\end{equation*}

\noindent Putting this back into \eqref{exhangeint} we see that

$$\int_{S_W} \int_G \ip{x}{gv}^{2k} \hspace{.5mm} dg \hspace{1mm} dx \leq \left(\frac{2k}{d_w}\right)^k\int_G \ip{v}{v}^k \hspace{.5mm} dg.$$

\noindent This integral is independent of $g$ and therefore it is equal to $\ip{v}{v}^k.$ The lemma now follows.

\end{proof}

\noindent Now we are ready to prove Theorem \ref{mainbound}.


\begin{proof}[Proof of Theorem \ref{mainthm}.]
From Lemma \ref{volumeint} we know that:
$$\left(\frac{\textnormal{Vol} \hspace{.5mm}B^{\po}_W(v)}{\textnormal{Vol} \hspace{.5mm} \Sigma_W} \right)^{1/d_w} \geq\left(\int_{S_W} \normi{\hspace{.5mm}l_x} \hspace{.5mm} dx\right)^{-1}.$$
We know that for any linear functional $l$ we have $\normi{\hspace{.3mm} l} \leq \alpha_k \norm{\hspace{.3mm}l}_{2k}$. Therefore we have
$$\left(\frac{\textnormal{Vol} \hspace{.5mm}B^{\po}_W(v)}{\textnormal{Vol} \hspace{.5mm} \Sigma_W} \right)^{1/d_w} \geq \alpha_k^{-1} \left(\int_{S_W} \norm{\hspace{.5mm}l_x}_{2k} \hspace{.5mm} dx\right)^{-1}=\alpha_k^{-1}\int_{S_W} \left(\int_G \ip{x}{gv}^{2k} \hspace{.5mm} dg \right)^{\frac{1}{2k}} dx.$$
We are now done by applying Lemma \ref{2k-moments}.
\end{proof}

\section{Application to Convex Forms.}
We need to establish that the cone of nonnegative bi-forms fits into our framework of duals of orbitopes and see how the cone of convex forms can be seen as a section.

\subsection{Nonnegative bi-forms as a Dual of an Orbitope.}\label{biformsasorbitope}
Let $Pos_{bi}$ be the convex cone of nonnegative bi-homogeneous forms in $Bi_{2n,2d}$:

$$Pos_{bi}=\left\{f \in Bi_{2n,2d} \st f(x,y) \geq 0 \foral x \in \mathbb{R}^{2n} \right\}.$$

To talk about the volume of $Pos_bi$ we first take a compact section of the cone. Recall that we have $SO(n) \times SO(n)$ acting on $Bi_{2n,2d}$ with each $SO(n)$ rotating the $x$ and $y$ coordinates separately. Let $M'_{bi}$ be the $SO(n) \times SO(n)$ invariant hyperplane $M'_{bi}$ of bi-forms of integral $1$ on the unit sphere:

$$M'_{bi}=\left\{f \in Bi_{2n,2d} \st \int_{\spht} f \hsp d \sigma=1\right\}.$$ Let ${Pos}'_{bi}$ be the section of $Pos_{bi}$ with $M'_{bi}$:

$$Pos'_{bi}=Pos_{bi} \cap M'_{bi}.$$

Let $M_{bi}$ be the linear hyperplane of bi-forms of integral $0$ on $\spht$:

$$M_{bi}=\left\{f \in Bi_{2n,2d} \st \int_{\spht} f \hsp d \sigma=0\right\}.$$

\noindent We translate $Pos'_{bi}$ into  $M_{bi}$  by subtracting the $SO(n) \times SO(n)$ invariant form $(x_1^2+\ldots x_n^2)^{d-1}(y_1^2+\ldots+y_n^2)$. Let $\widetilde{Pos}_{bi}$ be the translated section:

$$\widetilde{Pos}_{bi}=\left\{f \in Bi_{2n,2d} \st \int_{\spht} f \hsp d\sigma =0 \hspace{2.5mm} \textnormal{and} \hspace{2.5mm} f+(x_1^2+\ldots x_n^2)^{d-1}(y_1^2+\ldots+y_n^2) \in Pos_{bi}\right\}.$$

In other words $\widetilde{Pos}_{bi}$ consists of all bi-forms of integral $0$ on $\spht$ whose minimum on $\spht$ is at least -1.

We use the following inner product on $Bi_{2n,2d}$:

$$\ip{f}{g}=\int_{\spht} fg \hsp d\sigma.$$
\noindent Let $L_{x,y}$ be the bi-form in $M_{bi}$ such that $$\ip{f}{L_{x,y}}=f(x,y)$$
for all $f \in M_{bi}$. Let $v=L_{x,y}$ and consider the orbitope of $v$ under the action of $G=SO(n) \times SO(n).$ It follows from our definition of $L_{x,y}$ that $\widetilde{Pos}_{bi}$ is actually negative of the dual of $B^\po(v)$:

$$\widetilde{Pos}_{bi}=-B^\po(L_{x,y}).$$
Therefore we will be able to apply our volume bound for sections of duals of orbitopes to $\widetilde{Pos}_{bi}$. We now need to identify convex forms with such a section.

\subsection{Convex Forms as a Section of $\widetilde{Pos}_{bi}$.}\label{sectionconvexascut}
Recall that $K_{n,2d}$ is the convex cone of forms. As with bi-forms to estimate the size $K_{n,2d}$ we begin by intersecting with the hyperplane of forms of integral 1 on the unit sphere and translating the compact section to linear hyperplane of forms of integral zero.

Let $M_{n,2d}$ be the hyperplane of forms of integral zero on $\sph$ and let $\widetilde{K}_{n,2d}$ be the translated section:

$$\widetilde{K}_{n,2d}=\{p \in M_{n,2d} \st p+r^{2d} \in K_{n,2d}\},$$

\noindent where $r^{2d}=(x_1^2+\ldots+x_n^2)^d$.

Recall that to a form $p \in P_{n,2d}$ we associate a homogeneous bi-form $B_p \in Bi_{2n,2d}$ as follows:

$$B_p=y^TH_p \hsp y,$$

\noindent where $H_p$ is the Hessian of $p$.

We need to make sure forms in that have integral $0$ on $\sph$ get mapped to bi-forms of integral $0$ on $\spht$. Let $p$ be a form in $M_{n,2d}$ and consider the integral on $\spht$ of the associated bi-form $B_p$:

$$\int_{\spht} y^TH_p \hsp y \hsp d\sigma.$$
When integrating over $\spht$ lets integrate over $y$ first. In this case we have a quadratic form $y^TH_py$ in $y$ and then the integral of this form on the unit sphere is equal to the trace of $H_p$. Therefore we see that $$\int_{\spht} y^TH_p \hsp y \hsp d\sigma=\int_{\sph} \textnormal{tr}(H_p) \hsp d\sigma.$$
We observe that the trace of $H_p$ is simply the Laplacian $\Delta p$ of p:

$$\Delta p=\sum_{i=1}^n \frac{\partial^2 p}{\partial x_i^2}.$$

Using invariance properties of the Laplacian it is not has to show that if $\int_{\sph} p \hsp d\sigma =0$ then $\int_{\sph} \Delta p \hsp d\sigma=0$.

Lets take a closer look at the forms that lie in $\widetilde{K}_{n,2d}$. A form $p$ lies on the boundary of $K_{n,2d}$ if and only if the associated form $B_p$ has minimum of zero on $\sph \times \sph$. We need to be careful about subtracting $r^{2d}$ because the form $B_{r^{2d}}$ is not constant on $\sph \times \sph$.

Its not hard to calculate that
$$B_{r^{2d}}(x,y)=2d(x_1^2+\ldots +x_n^2)^{d-2}\left(2(d-1)\ip{x}{y}^2+(x_1^2+\ldots+x_n^2)(y_1^2+\ldots+y_n^2)\right).$$

For $(x,y) \in \sph \times \sph$ the form $B_{r^{2d}}$ simplifies to:
$$B_{r^{2d}}(x,y)=2d(2(d-1)\ip{x}{y}^2+1).$$

Therefore, if $p$ is in the boundary of $K_{n,2d}$ then the minimum of $B_{p}-B_{r^{2d}}$ on $\sph \times \sph$ is at most $-2d$ and at least $-2d(2d-1)$.

Now let $X_{n,2d}$ be the set forms $p$ of integral zero on $\sph$ such that the minimum of the associated form $B_p$ is at least $-1$ on $\sph \times \sph$:

$$X_{n,2d}=\left\{p \in M_{n,2d} \st \min_{(x,y)\in \sph \times \sph} B_p(x,y) \geq -1\right\}.$$

\noindent We can think of $X_{n,2d}$ as the section of $\widetilde{Pos}_{bi}$ with the linear subspace of bi-forms that come from Hessians of forms in $P_{n,2d}$. It follows from above that

$$2dX_{n,2d}\subset \widetilde{K}_{n,2d} \subset 2d(2d-1)X_{n,2d}.$$

\noindent Therefore it suffices to bound the volume of $X_{n,2d}$. In order to apply the bound of Theorem \ref{mainthm} we need to find the norm of $v=L_{x,y}$ and select the proper value of $k$ to use in the bound. This is done below.

\subsection{Length of $L_{x,y}$.}
\begin{lemma}\label{irred}
Let $W$  be an irreducible subspace of $Bi_{2n,2d}$ of dimension $D_W$ under the action of $SO(n)\times SO(n)$. For $(x,y)\in \spht$ let $L_{x,y}$ be a bi-form in $W$ such that $\ip{L_{x,y}}{f}=f(x,y)$ for all $f \in W$. Then the norm of $L_{x,y}$ is given by:
\begin{equation*}
\ip{L_{x,y}}{L_{x,y}}=D_{W}.
\end{equation*}
\end{lemma}
\begin{proof}
From the definition of $L_{x,y}$ we know that: $$L_{x,y}(x,y)=\ip{L_{x,y}}{L_{x,y}}=\int_{\spht} L_{x,y}^2 \hspace{.5mm} d\sigma.$$

\noindent Also, for any $g \in SO(n)\times SO(n)$
$$\ip{L_{x,y}}{gL_{x,y}}=L(g^{-1}(x,y)).$$

By invariance of $\spht$ under the action of $SO(n) \times SO(n)$ we can rewrite $\ip{L_{x,y}}{L_{x,y}}$ as an integral over $SO(n) \times SO(n)$:
\begin{equation*}
\ip{L_{x,y}}{L_{x,y}}=\int_{\spht} L_{x,y}^2 \hspace{.5mm} d\sigma=\int_{SO(n)\times SO(n)} \ip{L_{x,y}}{gL_{x,y}} \hspace{.5mm} d\mu(g).
\end{equation*}
Now we apply Lemma 6 of \cite{sasha3} and it follows that:
\begin{equation*}
\ip{L_{x,y}}{L_{x,y}}=\frac{\ip{L_{x,y}}{L_{x,y}}^2}{D_W}.
\end{equation*}
The Lemma now follows.
\end{proof}
We can now prove an identical statement for an invariant subspace, by splitting it into irreducible ones.

\begin{cor}\label{Lxy}
Let $W$ be an invariant subspace of $Bi_{2n,2d}$. For $(x,y)\in \spht$ let $L_{x,y}$ be a bi-form in $W$ such that $\ip{L_{x,y}}{f}=f(x,y)$ for all $f \in W$. Then the norm of $L_{x,y}$ is given by:
\begin{equation*}
\ip{L_{x,y}}{L_{x,y}}=D_{W}.
\end{equation*}
\end{cor}
\begin{proof}
We can write as an orthogonal sum of irreducible subspaces $W_i$: $W=\bigoplus W_i$. Let $L^i_{x,y}$ be the orthogonal projection of $L_{x,y}$ into $W_i$. It follows that $$L_{x,y}=\sum_i L^i_{x,y}$$ and $L^i_{x,y}$ are pairwise orthogonal. From pairwise orthogonality it follows that
$$\ip{L_{x,y}}{L_{x,y}}=\sum \ip{L^i_{x,y}}{L^i_{x,y}}.$$
From Lemma \ref{irred} we know that $\ip{L^i_{x,y}}{L^i_{x,y}}=\dim W_i$ and therefore $$\ip{L_{x,y}}{L_{x,y}}=\sum_i \dim W_i=D_W.$$
\end{proof}

\subsection{Establishing the right value of $k$.}


\begin{lemma}\label{choicek}
Let $l$ be a linear functional on $Bi_{2n,2d}$. Consider $l$ as a function on $G$ by setting $$l(g)=l(gL_{x,y}).$$ \noindent Then for $k \geq n \ln(2d+1)$ the $L^{2k}$ norm of $l$ approximates $L^{\infty}$ norm of $l$ to within a constant factor:
\begin{equation*}
\normi{\hspace{.3mm} l} \leq \alpha \norm{\hspace{.3mm}l}_{2k},
\end{equation*}
for some absolute constant $\alpha$. In particular $\alpha \leq 9e^2$.
\end{lemma}
\begin{proof}
Let $D_k$ be the dimension of of the span of the $k$-th tensor power $L_{x,y}^{\otimes k}$ of $L_{x,y}$. By Corollary 2 of \cite{sasha3} we know that
$$\norm{\hspace{.3mm}l}_{2k} \geq \left(D_k\right)^{\frac{1}{2k}} \normi{\hspace{.3mm} l}.$$

Taking $k$-tensor power of $L_{x,y}$ is dual to taking the $k$-th tensor power of its linear functional $l_{x,y}$. Since $l_{x,y}$ acts on $f \in Bi_{2n,2d}$ by evaluating it at $(x,y)$, the tensor power $l_{x,y}^{\otimes k}$ acts on $f^{\otimes k}$ by evaluating it at $(x,y)$ and then raising the result to $k$-th power:

$$l_{x,y}^{\otimes k}(f^{\otimes k})=f^k(x,y).$$

This is the same as taking $f^k$ and evaluating it at $(x,y)$. Therefore we see that the orbit of $L_{x,y}^{\otimes k}$ lies in the subspace $U_k$ of the $k$-th symmetric power that
consists of bi-homogeneous forms that have degree $2k$ in $y$ and $(2d-2)k$ in $x$. Therefore we have a formula for the dimension of $U_k$:
\begin{equation*}
\dim U_k=\binom{n+2k-1}{2k}\binom{n+(2d-2)k-1}{(2d-2)k} < \binom{n+2kd-1}{2kd}^2.
\end{equation*}
Now we just need to show that
$$\binom{n+2kd-1}{2kd}^{\frac{1}{2k}} \leq 3e,$$
\noindent for $k \geq \ln(2d+1).$

Let $H(x)$ be the entropy function for $0 \leq x \leq 1$:

$$H(x)=x \ln \frac{1}{x}+(1-x)\ln \frac{1}{1-x}.$$

\noindent The proof is finished by applying the following inequality:

$$\binom{a}{b} \leq \exp\left\{aH(b/a)\right\},$$

\noindent see for example Theorem 1.4.5 of \cite{entropy}.

\end{proof}

\subsection{Volume Bound for Convex Forms.}

We are now ready to prove the volume bound for the section of the cone of convex forms $\widetilde{K}_{n,2d}$ stated in Theorem \ref{convexthm}.

\begin{proof}[Proof of Theorem \ref{convexthm}.]
Recall from section \ref{sectionconvexascut} that $X_{n,2d}$ is the convex set in $M_{n,2d}$ consisting of forms $p$ such that the associated bi-form $B_p=y^TH_py$ has minimum at least -1 on $\spht$:

$$X_{n,2d}=\left\{p \in M_{n,2d} \st \min_{(x,y)\in \sph \times \sph} B_p(x,y) \geq -1\right\}.$$

\noindent We have shown above that $$2dX_{n,2d}\subset \widetilde{K}_{n,2d} \subset 2d(2d-1)X_{n,2d}.$$ Therefore it suffices to prove the following bound on the volume of $X_{n,2d}$:

\begin{equation*}
\left(\frac{\textnormal{Vol}\hsp X_{n,2d}}{\textnormal{Vol} \hsp \Sigma_H}\right)^{1/D_M} \geq \frac{1}{9e^2 \hsp \sqrt{2n\ln (2d+1)}}\sqrt{\frac{D_M}{D_{bi}}}.
\end{equation*}
Let $W$ be the linear subspace of $M_{bi}$ consisting to bi-forms coming from Hessians of forms in $P_{n,2d}$:

$$W=\left\{b \in M_{bi} \st b=y^TH_py \hspace{.7cm} \text{for some} \hspace{7mm} p \in P_{n,2d}\right\} \hspace{1cm}.$$

If we consider the associated bi-forms $B_p$ corresponding to $p\in X_{n,2d}$ then we can think of $X_{n,2d}$ as the section of $Pos_{bi}$ with $W$. We know from Section \ref{biformsasorbitope} that $Pos_{bi}$ is the negative of the dual of orbitope of $L_{x,y}$. Therefore we can apply Theorem \ref{mainthm} to the section of $Pos_{bi}$ with $W$. We find that

\begin{equation*}
\left(\frac{\textnormal{Vol} \hspace{.5mm}X_{n,2d}}{\textnormal{Vol} \hspace{.5mm} \Sigma_H} \right)^{1/D_M} \geq \alpha_k^{-1} \sqrt{\frac{D_M}{2k\ip{L_{x,y}}{L_{x,y}}}},
\end{equation*}
for some choice of $k$ and the corresponding $\alpha_k$.

We know from Corollary \ref{Lxy} that $\ip{L_{x,y}}{L_{x,y}}=\dim M_{bi} < \dim Bi_{2n,2d}=D_{bi}$ and from Lemma \ref{choicek} that for $k \geq n \ln (2d+1)$ we can take $\alpha_k=9e^2$. The Theorem now follows.
\end{proof}

\subsection{Relationship Between the Hessian and $L^2$ Metrics.}\label{relmetrics}
Our goal in this section is to show that the unit ball in the Hessian metric is not much smaller than the unit ball in the $L^2$ metric. We will actually show the following proposition that states that for any form in $P_{n,2d}$ the Hessian norm is not much larger than the $L^2$ norm which immediately implies the corresponding statement for unit balls:

\begin{prop} \label{metricswitch}
Let $g$ be a form in $P_{n,2d}$. Then
\begin{equation*}
\ip{g}{g}_H \leq \frac{12d^2(4d+n)^2}{n(n+2)}\ip{g}{g}_2.
\end{equation*}
\end{prop}
\noindent The Proposition implies that $$\left(\frac{\textnormal{Vol} \hsp \Sigma_2}{\textnormal{Vol} \hsp \Sigma_H}\right)^{1/D_M} \leq \frac{12d^2(4d+n)^2}{n(n+2)}.$$
\noindent We note that the constant of proportionality $$\frac{12d^2(4d+n)^2}{n(n+2)}$$ is clearly bounded for fixed degree $2d$.

Before proving Proposition \ref{metricswitch} we will need a couple of preliminary lemmas.

\begin{lemma}\label{vomit}
Let $g$ be a form in $n$ variables of degree $k$. Then $$\int_{\sph} \ip{\nabla g}{\nabla g} \hsp d\sigma \leq (2k^2+kn)\int_{\sph}g^2 \hsp d\sigma.$$
\end{lemma}
\begin{proof}
We observe that $\int_{\sph} \ip{\nabla g}{\nabla g} \hsp d\sigma$ and $\int_{\sph}g^2 \hsp d\sigma$ are both $SO(n)$-invariant positive definite quadratic forms on $P_{n,k}$. It follows that it is enough to check the inequality over the irreducible subspaces of $P_{n,k}$.

The forms in irreducible subspaces have form $(x_1^2+\ldots+x_n^2)^mf$ for some $m$ with $2m \leq k$ and a harmonic form $f$ of degree $k-2m$. Therefore we may assume that $g=(x_1^2+\ldots+x_n^2)^mf$ with $f$ harmonic.

In this case $$\frac{\partial g}{\partial x_i}=\frac{\partial f}{\partial x_i}(x_1^2+\ldots+x_n^2)^m+2mx_i(x_1^2+\ldots+x_n^2)^{m-1}f.$$

\noindent It follows that on the unit sphere $\sph$
$$\ip{\nabla g}{\nabla g}=\ip{\nabla f}{\nabla f}+4m(k-m)f^2.$$
On the unit sphere $g=f$ and we also know that $2m \leq k$, thus we see that $$\ip{\nabla g}{\nabla g}=\ip{\nabla f}{\nabla f}+4m(k-m)g^2 \leq \ip{\nabla f}{\nabla f}+2k^2g^2.$$

Therefore $$\int_{\sph} \ip{\nabla g}{\nabla g} \hsp d\sigma \leq \int_{\sph} \ip{\nabla f}{\nabla f} \hsp d\sigma + 2k^2\int_{\sph}g^2 \hsp d\sigma.$$
Since $f$ is harmonic of degree $k-2m$ it can be shown by application of Stokes' formula that
$$\int_{\sph} \ip{\nabla f}{\nabla f}\hsp d\sigma=(k-2m)(2k-4m+n-2)\int_{\sph}f^2 \hsp d\sigma.$$
See \cite{Duo} for details. Again, since $g=f$ on the unit sphere we see that
$$\int_{\sph} \ip{\nabla f}{\nabla f}\hsp d\sigma=(k-2m)(2k-4m+n-2)\int_{\sph}g^2 \hsp d\sigma \leq k(2k+n)\int_{\sph}g^2 \hsp d\sigma.$$
\end{proof}

\begin{lemma}\label{quads}
Let $q=y^TM(x)y$ be a homogeneous bi-form in $Bi_{2n,2d}$. Then
\begin{equation*}
\ip{q}{q}=\frac{2}{n(n+2)}\int_{\sph} \ip{M(x)}{M(x)} \hsp dx + \frac{1}{n(n+2)} \int_{\sph} \textnormal{tr}^2M(x) \hsp dx.
\end{equation*}
\end{lemma}

\begin{proof}
By definition,
$$\ip{q}{q}=\int_{\spht} q^2 \hsp d\sigma=\int_{\spht} \left(y^TM(x)y\right)^2 \hsp dydx.$$

When integrating over $\spht$ lets integrate over $y$ first. In this case we are integrating a quadratic form $y^TMy$ with respect to $y$ and the matrix $M$ depends on $x$ only. It is easy to show that for quadratic forms
$$\int_{\sph} (y^TMy)^2 \hsp dy = \frac{2}{n(n+2)}\ip{M}{M}+\frac{1}{n(n+2)}\textnormal{tr}^2 M.$$

\end{proof}
Now we are ready to prove Proposition \ref{metricswitch}.

\begin{proof}[Proof of Proposition \ref{metricswitch}]
Just as in proof of Lemma \ref{vomit} we note that $\ipe_H$ and $\ipe_2$ define $SO(n)$ invariant positive definite quadratic forms on $P_{n,2d}$. Therefore it suffices to show the Proposition for forms in an irreducible subspace of $P_{n,2d}$. Thus we may assume that $g$ has the form $$g=(x_1^2+\ldots+x_n^2)^m f$$ where $f$ is a harmonic form of degree $2d-2m$.

By definition, $$\ip{g}{g}_H=\int_{\spht} (y^TH_gy)^2 \hsp d\sigma.$$ Since the trace of $H_g$ is the Laplacian of $g$, $$\textnormal{tr} (H_g) = \Delta g,$$ it follows by Lemma \ref{quads} that
\begin{equation}\label{retch}
\ip{g}{g}_H=\frac{2}{n(n+2)}\int_{\sph} \ip{H_g(x)}{H_g(x)} \hsp dx + \frac{1}{n(n+2)} \int_{\sph} \left(\Delta g(x)\right)^2 \hsp dx.
\end{equation}

Since $g$ has the form $(x_1^2+\ldots+x_n^2)^m f$ with harmonic $f$ of degree $2d-2m$, it is not hard to see that for all $x\in \sph$
$$\Delta g(x) = 2m\left(n+4d-2m-2\right) g(x).$$
It follows that for $x \in \sph$ $$\left(\Delta g(x)\right)^2 \leq 2d(4d+n) g(x).$$

\noindent Therefore
\begin{equation}\label{disgorge}
\frac{1}{n(n+2)} \int_{\sph} \left(\Delta g(x)\right)^2 \hsp dx \leq \frac{4d^2(4d+n)^2}{n(n+2)} \int_{\sph} g^2 \hsp d\sigma=\frac{4d^2(4d+n)^2}{n(n+2)} \ip{g}{g}_2.
\end{equation}

\noindent Now we will need to bound $$\frac{2}{n(n+2)}\int_{\sph} \ip{H_g(x)}{H_g(x)} \hsp dx.$$ Let $g_i$ be the the derivative of $g$ with respect to $x_i$:
$$g_i=\frac{\partial g}{\partial x_i}.$$
By summing over rows it is easy to see that $$\ip{H_g}{H_g}=\sum_{i=1}^n \ip{\nabla g_i}{\nabla g_i}.$$
Therefore
$$\int_{\sph} \ip{H_g}{H_g} \hsp d\sigma = \int_{\sph} \sum _{i=1}^n \ip{\nabla g_i}{\nabla g_i} \hspace{.5mm} d\sigma=\sum _{i=1}^n \int_{\sph}  \ip{\nabla g_i}{\nabla g_i} \hspace{.5mm}d\sigma.$$
Each $g_i$ is a homogeneous form of degree $2d-1$ and by Lemma \ref{vomit} it follows that
$$\int_{\sph}  \ip{\nabla g_i}{\nabla g_i} \hspace{.5mm}d\sigma \leq (8d^2+2dn) \int_{\sph} g_i^2 d\sigma.$$
Thus we see that $$\int_{\sph} \ip{H_g}{H_g} \hsp d\sigma \leq (8d^2+2dn) \int_{\sph} \sum_{i=1}^n g_i^2 \hspace{.5mm} d\sigma.$$
We observe that $$\sum_{i=1}^n g_i^2=\ip{\nabla g}{\nabla g}$$
and therefore we can apply Lemma \ref{vomit} again to see that
$$\int_{\sph} \ip{H_g}{H_g} \hsp d\sigma \leq (8d^2+2dn)^2 \int_{\sph} g^2 \hspace{.5mm} d\sigma=4d^2(2d+n)^2\ip{g}{g}_2.$$
Plugging this back into \eqref{retch} and combining with \eqref{disgorge} the Proposition follows.
\end{proof}

\end{document}